\documentclass[12pt,a4paper,psamsfonts]{amsart}
\usepackage[english]{babel}
\usepackage[T1]{fontenc}
\usepackage{fancyhdr}
\usepackage{appendix}
\usepackage{amssymb,amscd,amsxtra,calc}
\usepackage{mathrsfs}
\usepackage{cmmib57}
\usepackage{multirow}
\usepackage[all]{xy}
\usepackage{longtable}
\usepackage[colorlinks=true,linkcolor=blue,anchorcolor=blue,citecolor=blue,pagebackref,linktocpage]{hyperref}
\usepackage{cleveref}
\usepackage{tikz}
\usepackage{cite}
\theoremstyle{plain}
    \newtheorem{thm}{Theorem}[section]

     \newtheorem{conjecture}[thm]{Conjecture}
    
    \newtheorem{example}[thm]{Example}
    
    \newtheorem{proposition}[thm]{Proposition}
    \newtheorem{question}[thm]{Question}
     \newtheorem{problem}[thm]{Problem}
    \newtheorem{theorem}[thm]{Theorem}

\theoremstyle{definition}
    \newtheorem{definition}[thm]{Definition}

    \newtheorem*{notation*}{Notation and Terminology}
      
    \newtheorem{remark}[thm]{Remark}
    
\theoremstyle{remark}

\newcommand{\bP}{\mathbb{P}}
\newcommand{\bQ}{\mathbb{Q}}
\newcommand{\bR}{\mathbb{R}}
\newcommand{\bC}{\mathbb{C}}

\newcommand{\bZ}{\mathbb{Z}}

\newcommand{\Aut}{\operatorname{Aut}}

\newcommand{\id}{\operatorname{id}}

\newcommand{\NE}{\overline{\operatorname{NE}}}
\newcommand{\Nef}{\operatorname{Nef}}

\newcommand{\Pic}{\operatorname{Pic}}

\newcommand{\mstriangle}[1]{
\begin{tikzpicture}[x=0.3cm,y=0.3cm]
\draw (-0.4,-0.433) -- (1.4,-0.433);
\draw (-0.2,-0.7794) -- (0.7,0.7794);
\draw (1.2,-0.7794) -- (0.3,0.7794);
\end{tikzpicture}
}
\newcommand{\mssharp}[1]{
\begin{tikzpicture}[x=0.3cm,y=0.3cm]
\draw (-0.8,-0.5) -- (0.8,-0.5);
\draw (-0.8,0.5) -- (0.8,0.5);
\draw (-0.5,-0.8) -- (-0.5,0.8);
\draw (0.5,-0.8) -- (0.5,0.8);
\end{tikzpicture}
}

\makeatletter

\newcommand{\Rmnum}[1]{\expandafter\@slowromancap\romannumeral #1@}
\makeatother

\begin{document}
\title[Bounded cohomology property]
{Smooth projective surfaces with bounded cohomology property}
\author{Ziyu Hua}
\address{
School of Mathematics, East China University of Science and Technology, Shanghai 200237, P. R. China}
\email{\href{mailto:y30251353@mail.ecust.edu.cn}{y30251353@mail.ecust.edu.cn}}
\author{Sichen Li}
\address{
School of Mathematics, East China University of Science and Technology, Shanghai 200237, P. R. China}
\email{\href{mailto:sichenli@ecust.edu.cn}{sichenli@ecust.edu.cn}}

\begin{abstract}
In this paper, we first prove that every Mori dream surface $X$ satisfies the bounded cohomology property (BCP for short).
Namely, there exists a constant $c_X>0$ such that $h^1(\mathcal O_X(C))\le c_Xh^0(\mathcal O_X(C))$ for every curve $C$ on $X$.
We then prove that there is a positive constant $m(Y)$ such that $l_C:=(K_Y\cdot C)(C^2)^{-1}\le m(Y)$ for every ample curve $C$ on a geometrically ruled surface $Y$ over a curve of genus $g$, and $Y$ satisfies the BCP if $g\le1$.
\end{abstract}
\maketitle
\section{Introduction}
The Bounded Negativity Conjecture (BNC for short)  is one of the most intriguing problems in the theory of projective surfaces and can be formulated as follows.
\begin{conjecture}
\cite[Conjecture 1.1]{Bauer et al. 2013} 
For a smooth projective surface $X$ over $\bC$, there exists an integer $b(X)\ge0$ such that $C^2\ge-b(X)$ for every curve $C\subseteq X$.
\end{conjecture}
\begin{definition}
(cf. \cite[Conjecture 2.5.3]{Bauer et al. 2012})
\label{defn-BCP}
A smooth projective surface $X$ is said to satisfy the bounded cohomology property (BCP for short) if there exists a constant $c_X>0$ such that $h^1(\mathcal O_X(C))\le c_Xh^0(\mathcal O_X(C))$ for every curve $C$ on $X$.
\end{definition}
Ciliberto et al. \cite{Ciliberto et al. 2017} first showed that the BCP implies the BNC.
It is well-known that the SHGH Conjecture (cf. \cite[Conjecture 2.5.1]{Bauer et al. 2012}) implies Nagata's Conjecture \cite{Nagata59}, which is motivated by Hilbert's 14th problem (cf. \cite[Lemma 2.4]{CHMR13}).
For example, it is open for a weak version of the SHGH Conjecture: specifically, how to establish the BCP for the blow-up of $\bP^2$ at generic points $p_1,\cdots, p_n$ when $n\ge10$?	 	

Thus, a fundamental problem is to characterize surfaces with the BCP as follows.
\begin{problem}
\label{BCP-Prob}
(cf. \cite[Question 6]{Ciliberto et al. 2017})
Classify all smooth projective surfaces with the BCP.
\end{problem} 
As the first address to Problem \ref{BCP-Prob}, we note the following result.
\begin{proposition}
\label{prop-rational}
A smooth projective surface $X$ satisfies the BCP provided that the anti-canonical divisor $-K_X$ is pseudoeffective.
\end{proposition}
To initiate the study of the BCP for a surface $X$, Ciliberto et al. showed in \cite[Proposition 15]{Ciliberto et al. 2017} that the BCP implies the uniform boundedness of $l_C$ (cf. Definition \ref{Defn-l_C}).
Conversely, given the uniform boundedness of $l_C$, the second author \cite{Li21} provided a numerical characterization of the BCP as follows.
\begin{proposition}
\label{Prop-Numerical}
\cite[Proposition 2.3]{Li21}
Let $X$ be a smooth projective surface.
If $X$ satisfies the BNC and there exists a positive constant $m(X)$ such that $l_C\le m(X)$ for every curve $C$ on $X$ and $D^2\le m(X)h^0(\mathcal O_X(D))$ for every curve $D$ with $D^2>0$ on $X$, then $X$ satisfies the BCP.
\end{proposition}
\begin{remark}
\label{mX-rmk}
Compared with~\cite[Proposition~2.3]{Li21}, we omit the hypothesis $l_D>1$ for the sake of simplifying our statement.
Indeed, if $l_D\le 1$, then $(K_X-D)\cdot D\le 0$. 
Riemann-Roch theorem and Proposition \ref{Serre} imply that 
\begin{equation*}
\begin{split}
		2h^1(\mathcal O_X(D))&=2h^0(\mathcal O_X(D))+(K_X-D)\cdot D+2h^2(\mathcal O_X(D))-2\chi(\mathcal O_X)
		\\&\le (2q(X)+2)h^0(\mathcal O_X(D)).
\end{split}
\end{equation*}
Once the uniform boundedness of \(l_C\) is established, we sometimes verify \(h^1(\mathcal{O}_X(D)) \le m(X)h^0(\mathcal{O}_X(D))\) directly in place of \(D^2 \le m(X)h^0(\mathcal{O}_X(D))\) for every curve $D$ with $D^2>0$ on $X$.
\end{remark}
Let $X$ be a smooth projective surface with Picard number $\rho(X)=2$ and $b(X)>0$.
Motivated by a classification result of such $X$ with two negative curves (cf. \cite[Claim 2.11]{Li19}), the second author (in \cite{Li21, Li23}) established the BCP for  $X$ if either (i) $\kappa(X)\le1$, or (ii) $\kappa(X)=2$, the irregularity $q(X)=0$ and $X$ has a semiample curve $C$ with $C^2=0$.
Moreover, the second author \cite{Li23} asked the following question.
\begin{question}
(cf. \cite[Question 1.3]{Li23})
Does every smooth projective surface $X$ satisfy the BCP provided that the closed Mori cone $\NE(X)$ is rational polyhedral?
\end{question}
In \cite{HK00}, Hu and Keel proved that if the Cox ring of an algebraic variety $X$ is finitely generated, then $X$ enjoys ideal properties in the aspect of the minimal model program.
Such varieties are called Mori dream spaces.
As the name suggests, Mori dream spaces possess many desirable properties and intricate geometry.
Note that $\NE(X)$ is rational polyhedral when $X$ is a Mori dream surface.
See \cite{ADHL15} and references therein for more details.

A few years ago, Joaquim Ro\'e asked the second author whether the Mori dream surfaces satisfy the BCP as follows.
\begin{question}
\cite[Question 1.6]{Li23}
\label{Que-MDS}
Does every Mori dream surface satisfy the BCP?
\end{question}
In this paper, we completely answer Question \ref{Que-MDS} as follows.
\begin{theorem}
\label{thm-Main}
Every Mori dream surface satisfies the BCP.
\end{theorem}
By Proposition \ref{Prop-Numerical}, a primary strategy to establish the BCP is to show the uniform boundedness of $l_C$.
Now we ask whether it is possible that a smooth projective surface $X$ still admits the uniform boundedness of $l_C$ if  $\NE(X)$ is not rational polyhedral.
This question has an affirmative answer for all geometrically ruled surfaces as follows.
 \begin{theorem}
\label{uniform-thm}
Let $X$ be a geometrically ruled surface with invariant $e$ over a curve of genus $g$.
Then the following statements hold.
\begin{enumerate}
	\item $X$ admits the uniform boundedness of $l_C$.
	\item  If $g\le1$, then $X$ satisfies the BCP.
	\item  If $g\ge 2$ and $e>0$, then $X$ satisfies the BCP.
\end{enumerate}
\end{theorem}	
\begin{remark}
To completely settle Problem \ref{BCP-Prob} for all geometrically ruled surfaces $X$ with invariant $e$ over a curve of genus $g$, it is sufficient to consider the case where $g\ge2$ and $e\in [-g,0]$, in view of \cite[Theorem 1]{Nagata70} and Theorem \ref{uniform-thm}.
\end{remark}
The paper is organized as follows.
In Section \ref{Pre},  we organize preliminary material into three subsections: definitions and preliminary results,  independent auxiliary results, and the proof of Proposition \ref{prop-rational}.
We prove Theorem \ref{thm-Main} in Section \ref{Sect-Main}, and Theorem \ref{uniform-thm} in Section \ref{Sect-uniform}.
Finally, we provide examples of surfaces with the BCP in Section \ref{Sect-Example}.
\section{Preliminaries}
\label{Pre}
In this section, we collect the basic definitions and preliminary results required for proving Theorems \ref{thm-Main} and \ref{uniform-thm}, present several independent auxiliary results, and give a proof of Proposition \ref{prop-rational}.

{\bf Notation and Terminology.}
In this paper, $X$ is a smooth projective surface over $\bC$.
\begin{itemize}
\item By a curve on $X$, we mean a reduced and irreducible curve.
\item A negative curve on $X$ is a curve with negative self-intersection.
\item A $(-k)$-curve on $X$ is a curve $C$ with $C^2=-k<0$.
 \item A prime divisor $C$ on $X$ is either a nef curve or a negative curve (in the latter case, $h^0(\mathcal O_X(C))=1$).
\item We say $X$ has $b(X)>0$ if $X$ contains at least one negative curve.
\item We say $X$ has $b(X)=0$ if $X$ contains  no negative curves.
 \item For every $\bR$-divisor $C$ with $C^2\ne0$ on $X$, we define a value $l_C$ associated to $C$ as follows:
\begin{equation*}
                                                       l_C:=\frac{(K_X\cdot C)}{\max\bigg\{ 1, C^2\bigg\}}.
\end{equation*}
\item For every $\bR$-divisor $C$ with $C^2=0$ on $X$, we define a value $l_C$ associated to $C$ as follows:
\begin{equation*}
                                   l_C:=\frac{(K_X\cdot C)}{\max\bigg\{1,h^0(\mathcal O_X(C))\bigg\}}.
\end{equation*}
\end{itemize}
\subsection{Definitions and preliminary results}
\begin{definition}
 (cf. \cite[Remark 2.4]{Li23})
 \label{Defn-l_C}
We say a smooth projective surface $X$ admits uniform boundedness of $l_C$ if there exists a positive constant $m(X)$ such that $l_C\le m(X)$ for every curve $C$ on $X$.
\end{definition}
\begin{definition}
We say the closed Mori cone \(\NE(X)\) is rational polyhedral if there exist finitely many curves $C_1,\cdots, C_s$ on $X$ such that $\NE(X)=\sum \bR_{\ge0}[C_i]$.
\end{definition}
\begin{definition}
\cite{HK00}
 A Mori dream surface is a smooth projective surface $X$ with finitely generated divisor class group $\mathrm{Cl}(X)$ and $q(X)=0$, and  $\Nef(X)=\mathrm{SAmp}(X)$, and $\NE(X)$ is rational polyhedral.
Here, $\mathrm{SAmp}(X)$ denotes the cone of semiample divisors on $X$.
\end{definition}

The following result is due to Serre duality.
\begin{proposition}\label{Serre}
Let $C$ be a curve on a smooth projective surface $X$.
Then
\begin{equation*} 
 h^2(\mathcal O_X(C))-\chi(\mathcal O_X)\le q(X)-1.	
\end{equation*}
\end{proposition}
 The following result is crucial in this paper.
\begin{proposition}
\cite[Proposition 3.4]{Li23}
\label{Prop-NE}
If $\NE(X)=\bR_{\ge0}[C_1]+\bR_{\ge0}[C_2]$ such that $C_1, $  and $C_2$ are some curves on a surface $X$.
Then $X$ admits the uniform boundedness of $l_C$.
\end{proposition}
\subsection{Independent auxiliary results}
\begin{proposition}
\label{Prop-Kappa0}
Let $X$ be a smooth projective surface with $\rho(X)=2$ and $\kappa(X)=0$.
Then $X$ satisfies the BCP.
\end{proposition}
\begin{proof}
We first assume that $K_X$ is nef.
Since $\kappa(X)=0$, it follows that $K_X\equiv0$.
Then $l_C=0$ for every curve $C$ on $X$.
The adjunction formula gives that $C^2\ge-2$.
Then, by the Riemann-Roch Theorem and Proposition \ref{Serre}, we have
$$
      h^1(\mathcal O_X(C))=h^0(\mathcal O_X(C))+h^2(\mathcal O_X(C))-\chi(\mathcal O_X)-\frac{C^2}{2}\le (q(X)+1)h^0(\mathcal O_X(C)).
$$
So $X$ satisfies the BCP.
Now we may assume that $X$ is a one-point blow-up of a surface.
Thus, $X$ satisfies the BCP by \cite[Theorem 1.5]{Li23}.
\end{proof}
Let $X$ be a smooth projective surface with $\rho(X)=2, \kappa(X)=1$ and $b(X)>0$.
By \cite[Lemma 3.2]{Li23}, $X$ satisfies the BCP.
Note that $\NE(X)$ is  rational polyhedral.
Moreover, we have the following result.
\begin{proposition}
\label{prop-kappa1}
Let $X$ be a smooth projective surface with  $\rho(X)=2$ and $\kappa(X)=1$.
If $\NE(X)$ is rational polyhedral, then $X$ satisfies the BCP.
\end{proposition}
\begin{proof}
Let $\NE(X)=\bR_{\ge0}[C_1]+\bR_{\ge0}[C_2]$ such that $C_1$ and $C_2$ are some curves on $X$.
Since $\kappa(X)=1,\rho(X)=2$, and $\kappa(X)$ is a birational invariant, we may assume that $K_X=C_1$ is nef.
By \cite[Lemma 3.2]{Li23}, we may assume that $C_2^2=0$.
Note that $X$ satisfies the BNC since $\rho(X)=2$.
By Propositions \ref{Prop-Numerical} and \ref{Prop-NE}, it suffices to show that there exists a constant $c_X>0$ such that for every nef curve $D$ with $D^2>0$ on $X$, $h^1(\mathcal O_X(D))\le c_Xh^0(\mathcal O_X(D))$.

Now we may take a nef and big curve $D=a_1C_1+a_2C_2$ with $a_1>0$ and $a_2>0$ by \cite[Proposition 3.1]{Li21}.
Note that $D\cdot C_1,D\cdot C_2\in\bZ_{>0}$ as $D$ is a curve.
So we may assume that $a_1,a_2\in\bZ_{>0}$ after replacing $D$ by $cD$ with some $c\in\bZ_{>0}$.
In fact, $c=(C_1\cdot C_2)^2+(C_1)^2(C_2)^2>0$.
If $a_1=1$, then $(K_X-D)D=-a_2C_2\cdot D=-a_2(C_1\cdot C_1)<0$.
By the Riemann-Roch Theorem and Proposition \ref{Serre}, we have
\begin{equation*}
\begin{split}
	 h^1(\mathcal O_X(D))&=h^0(\mathcal O_X(D))+\frac{(K_X-D)\cdot D}{2}
+h^2(\mathcal O_X(D))-\chi(\mathcal O_X)	 \\&\le (q(X)+1)h^0(\mathcal O_X(D)).
\end{split}	
\end{equation*}
If $a_1>1$, then  $D-K_X=(a_1-1)C_1+a_2C_2$ is big.
Note that $(D-K_X)C_1=a_2(C_1\cdot C_2)>0$ and $(D-K_X)C_2=(a_1-1+a_2)(C_1\cdot C_2)>0$.
Therefore, $D-K_X$ is nef and big.
Then 
$$
h^1(\mathcal O_X(D))=h^1(\mathcal O_X((D-K_X)+K_X)=0
$$
by the Kawamata-Viehweg vanishing theorem.
In all, $X$ satisfies the BCP.
\end{proof}
\subsection{Proof of Proposition \ref{prop-rational}}
Bauer et al. \cite[Proposition 3.1.3]{Bauer et al. 2012} proved the BCP for smooth projective rational surfaces with $-K_X$ effective, and we slightly generalize their result as follows.
\begin{proposition}
\label{Kod0}
Let $X$ be a smooth projective surface with anti-Kodaira dimension $\kappa(X,-K_X)\ge0$.
Then $X$ satisfies the BCP.
\end{proposition}
\begin{proof}
Let $C$ be a curve on $X$.
There exists a positive constant $m$ such that $-mK_X$ is linearly equivalent to an effective divisor $D$.
We first show that $X$ satisfies the BNC.
Then we may assume that $C$ is not a component of $D$ since $D$ has only finitely many components.
Then $-mK_X\cdot C=D\cdot C\ge0$.
By the adjunction formula to $C$, 
$$
C^2=2g(C)-2-(K_X\cdot C)\ge-2.
$$
Therefore, $X$ satisfies the BNC.
To show the BCP for $X$, we may also assume that $C$ is not a component of $D$.
We further note that
\begin{equation}
\label{-K_X,kappa<0,eq1}
(K_X-C)\cdot C\le 2.
\end{equation}
Riemann-Roch Theorem, (\ref{-K_X,kappa<0,eq1}) and Proposition \ref{Serre} imply that
\begin{equation*}\begin{split}
            h^1(\mathcal O_X(C))&=h^0(\mathcal O_X(C))+h^2(\mathcal O_X(C))+\frac{(K_X-C)\cdot C}{2}-\chi(\mathcal O_X)\\&\le (q(X)+1)h^0(\mathcal O_X(C)).
\end{split}\end{equation*}
Thus $X$ satisfies the BCP.
\end{proof}
\begin{proposition}
\label{-K_Xnef}
Let $X$ be a smooth projective surface with nef anticanonical divisor $-K_X$.
Then $X$ satisfies the BCP.
\end{proposition}
\begin{proof}
Note that $K_X\cdot C\le0$ for every curve $C$ on $X$ since $-K_X$ is nef.
By the adjunction formula to $C$, $C^2=2g(C)-2-(K_X\cdot C)\ge-2$.
We further note that
\begin{equation}
\label{-K_Xnef,eq1}
(K_X-C)\cdot C\le 2.
\end{equation}
The remaining proof is the same as the proof of Proposition \ref{Kod0}.
\end{proof}
\begin{proof}[Proof of Proposition \ref{prop-rational}]
It follows from \cite[Lemma 3.1]{Sakai84} and Proposition \ref{Kod0}.
\end{proof}
\section{Proof of Theorem \ref{thm-Main}}
\label{Sect-Main}
\begin{proposition}
\label{Prop-NE-Uniform}
Every Mori dream surface $X$ admits the uniform boundedness of $l_C$.
\end{proposition}
\begin{proof}
To show the uniform boundedness of $l_C$, by \cite[Proposition 2.8]{Li23}, we may take a nef curve $D$ with $D^2=0$.
Note that $\mathrm{Nef}(X)=\sum_{i=1}^k \bR_{\ge0}[C_i]$ and each $C_i$ is a semiample curve, since $X$ is a Mori dream surface.
If $C_i\ne rC_j$ for any $r\in\bQ_{>0}$, then $C_i\cdot C_j>0$ by \cite[Proposition 5.1.1.1]{ADHL15}.
As a result, for a nef curve $D=\sum_{i=1}^r a_iC_i$ with each $a_i>0$, we have $D^2>0$ if $r\ge2$.
So $D=a_kC_k$ with $a_k>0$.
Then there exists a positive constant $m(X)$ such that $l_D\le m(X)$ for every nef curve $D$ with $D^2=0$ by \cite[Proposition 2.7]{Li23}.
In all, $X$ admits the uniform boundedness of $l_C$.
\end{proof}
\begin{proof}[Proof of Theorem \ref{thm-Main}]
Note that $X$ satisfies the BNC since $\NE(X)$ is rational polyhedral.
By Propositions \ref{Prop-Numerical} and \ref{Prop-NE-Uniform} and Remark \ref{mX-rmk}, it suffices to show that there is a constant $c_X>0$ such that for every  nef curve $D$ with $D^2>0$ on $X$, we have either $h^1(\mathcal O_X(D))\le c_Xh^0(\mathcal O_X(D))$ or $D^2\le c_Xh^0(\mathcal O_X(D))$.

Note that  $\Nef(X)=\sum_{i=1}^k\bR_{\ge0}[C_i]$ and each $C_i$ is a semiample divisor, since $X$ is a Mori dream surface.
If $k=1$, then $\rho(X)=1$, and every nef divisor is ample.
Now we fix an ample curve $H$.
Therefore, we may assume $D = aH$ with $a > 0$.
Then there is a positive constant $b$ such that if $a>b$, then $D-K_X$ is ample.
As a result, $$h^1(\mathcal O_X(D))=h^1(\mathcal O_X(K_X+(D-K_X)))=0$$ by the Kodaira vanishing theorem.
If $a\le b$, then $(K_X-D)D\le b|K_X\cdot H|$.
So by the Riemann-Roch theorem and Proposition \ref{Serre}, we have
\begin{equation}
\begin{split}
	h^1(\mathcal O_X(D))&\le h^0(\mathcal O_X(D))+\frac{(K_X-D)\cdot D}{2}+h^2(\mathcal O_X(D))-\chi(\mathcal O_X)
	\\&\le \bigg(\frac{1}{2}b|K_X\cdot H|+q(X)+1\bigg)h^0(\mathcal O_X(D))
	\\&\le \bigg(\frac{1}{2}b|K_X\cdot H|+1\bigg)h^0(\mathcal O_X(D)),
\end{split}
\end{equation}
where we use $q(X)=0$ since $X$ is a Mori dream surface.
So $X$ satisfies the BCP.

Now we assume that $k\ge2$.
Take a semiample curve $D=\sum_{i=1}^r a_iC_i$ with each $a_i>0$.

Without loss of generality, we may assume that \(a_1\ge a_2\ge\cdots \ge a_r\) and \(a_1\gg 1\).
If \(a_1\) were bounded above by $k_X$, then 
\begin{equation*}
\begin{split}
D^2 &=\sum_{i=1}^r a_i^2C_i^2+\sum_{1\le i<j\le r}2(C_i\cdot C_j)
 \\ & \le  c_Xh^0(\mathcal O_X(D)),
\end{split}
\end{equation*}
where $$c_X:=k_X^2(r+r^2)\max\big\{C_i^2,\; C_i\cdot C_j \;\big|\; 1\le i\le r,\; 1\le i<j\le r\big\}.$$
 So we may assume that  \(a_1\gg 1\) to complete the proof.
 
Note that
\begin{equation}
\begin{split}
\label{MDS-eq1}
	(K_X-D)D&\le \sum_{i=1}^r a_i(K_X\cdot C_i)
	\\&\le \bigg(\sum_{i=1}^r|K_X\cdot C_i|\bigg)a_1.
\end{split}
\end{equation}
 By \cite[Corollary 2.1.38]{Lazarsfeld04} and $\kappa(X,C_1)=1$, for $a_1\gg1$, there exist two positive constants $c=(2c_X-2)^{-1}\sum_{i=1}^r|K_X\cdot C_i|$ and $c_X>1$  (which are independent of the choice of $D$) such that 
\begin{equation}
\label{MDS-eq2}
	 h^0(\mathcal O_X(D))\ge h^0(\mathcal O_X(a_1C_1))\ge\frac{1}{2c_X-2}\bigg(\sum_{i=1}^r|K_X\cdot C_i|\bigg)a_1.
\end{equation}
Then by Riemann-Roch Theorem, (\ref{MDS-eq1}), (\ref{MDS-eq2}) and Proposition \ref{Serre}, we have 
\begin{equation*}\begin{split}
	  h^1(\mathcal O_X(D))&=h^0(\mathcal O_X(D))+h^2(\mathcal O_X(D))+\frac{(K_X-D)\cdot D}{2}-\chi(\mathcal O_X)\\&\le(c_X+q(X))h^0(\mathcal O_X(D))
	  \\&\le c_Xh^0(\mathcal O_X(D)),
\end{split}
 \end{equation*}
 where we use $q(X)=0$ since $X$ is a Mori dream surface.
 In all, $X$ satisfies the BCP.
\end{proof}
Let $X$ be a smooth projective surface such that $\NE(X)$ is rational polyhedral and $q(X)=0$.
Use the same argument in the proof of Proposition \ref{Prop-NE-Uniform}, $X$ still admits the uniform boundedness of $l_C$. 
Motivated by Theorem \ref{thm-Main}, we ask the following question.
\begin{question}
Does every smooth projective surface $X$ with $q(X)=0$ satisfy the BCP if $\NE(X)$ is rational polyhedral?
\end{question}
 \section{Proof of Theorem \ref{uniform-thm}}
 \label{Sect-uniform}
From now on, we will consistently use the definitions and notation of geometrically ruled surfaces as presented in \cite[Chapter V.2]{Hartshorne77}.
\begin{definition}
\cite{Hartshorne77}
A geometrically ruled surface, or simply ruled surface, is a surface $X$, together with a surjective morphism $\pi: X\to C$ onto a nonsingular curve $C$, such that the fiber $X_y$ is isomorphic to $\bP^1$ for every point $y\in C$, and such that $\pi$ admits a section, i.e., a morphism $\sigma: C\to X$ such that $\pi\circ\sigma=\id_C$.
\end{definition}
\begin{proposition}
\label{Prop-ruled}
\cite[Propositions V.2.3 and V.2.9]{Hartshorne77}
Let $\pi: X\to C$ be a geometrically ruled surface with invariant $e$ over a smooth curve $C$ of genus $g$.
Let $C_0\subseteq X$ be a section such that $\mathcal O_X(C_0)\cong O_X(1)$ and $C_0^2=-e$, and let $f$ be a fibre.
Then we have the following result:
$$
	\Pic 	X\cong \bZ C_0\oplus\pi^*\Pic C, C_0\cdot f=1, f^2=0, \text{~and ~}~K_X\equiv -2C_0+(2g-2-e)f.
$$
\end{proposition}
\begin{proposition}
\label{e=0}
Let $\pi: X\to C$ be a geometrically ruled surface with invariant $e=0$ over a smooth curve $C$ with genus $g$.
Then $X$ admits the uniform boundedness of $l_C$.
In addition, if $g\le1$, then $X$ satisfies the BCP.
\end{proposition}
\begin{proof}
Note that $C_0^2=f^2=0$ as in Proposition \ref{Prop-ruled}.
Therefore, $\NE(X)=\bR_{\ge0}[C_0]+\bR_{\ge0}[f]$.
So $X$ admits the uniform boundedness of $l_C$ by Proposition \ref{Prop-NE}.

Now we assume that $g\le1$.
Then by Proposition \ref{Prop-ruled}, we have
$$-K_X\equiv 2C_0+(2-2g)f, \quad -K_X\cdot C_0=2-2g\ge0, \text{ and~ } -K_X\cdot f=2.$$
So $-K_X$ is nef.
By the proof of Proposition \ref{-K_Xnef}, for every curve $D$ on $X$, we have
\begin{equation*}
h^1(\mathcal O_X(D))\le (q(X)+1)h^0(\mathcal O_X(D)).
\end{equation*}
In conclusion, $X$ satisfies the BCP.
\end{proof}
\begin{proposition}
\label{e<0}	
Let $\pi: X\to C$ be a geometrically ruled surface with invariant $e<0$ over a smooth curve $C$. 
Then $X$ admits the uniform boundedness of $l_C$.
In particular, $X$ satisfies the BCP provided that $g=1$.
\end{proposition}
\begin{remark}
Let $X$ be a geometrically ruled surface with invariant $e<0$ over a smooth curve $C$ of genus $g\ge1$.
Note that $\NE(X)$ is rational polyhedral if and only if there is a curve $D_0\equiv a(C_0+(\frac{e}{2}))f$ with some $a\ge2$  by \cite[Theorem 3]{Rosoff02}.
When $g=1$, Homma \cite{Homma82} discovered that $X$ contains an elliptic curve $D$ numerically equivalent to $2C_0-f$.
Note that $\NE(X)$ may not be rational polyhedral if $g\ge2$ (cf. \cite[Example 1.23(1)]{KM98}).
\end{remark}
\begin{proof}[Proof of Proposition \ref{e<0}]
Take a curve $D=aC_0+bf$(where $D\ne C_0, f$) with $a,b\in\bZ$ as given in Proposition \ref{Prop-ruled}.
Notice that $C_0^2=-e>0$.
 By \cite[Proposition V. 2.21]{Hartshorne77}, we have the following result.
 \begin{enumerate}
 	\item $C_0$ is ample.
 	\item  either $a=1, b\ge0$ or $a\ge2, b\ge\frac{1}{2}ae$.
 	\item  $D$ is ample if and only if $a>0$ and $b>\frac{1}{2}ae$.
 	\item $b(X)=0$.
 	\item $D$ is nef but not ample if and only if $a\ge2$ and $b=\frac{1}{2}ae$.
 \end{enumerate}
Suppose $D_0=a_0C_0+(\frac{1}{2}a_0e)f$ with some $a_0\ge2$ is a curve.
Then $D_0^2=0$ and $\NE(X)=\bR_{\ge0}[D_0]+\bR_{\ge0}[f]$.
By Proposition \ref{Prop-NE}, $X$ admits the uniform boundedness of $l_C$.

Now we suppose that each curve $ D = aC_0 + bf$ (where $ D\ne C_0, f$) is ample.
Then $a>0, b>\frac{1}{2}ae$, $D\cdot C_0>0$, and $D\cdot f>0$.
 Take $C'=C_0+(\frac{1}{2}e)f$.
 Note that $D\cdot C'=b-\frac{1}{2}ae>0$.
We then compute $l_D$ (using $K_X\equiv -2C_0+(2g-2-e)f$):
\begin{equation*}
\begin{split}
	l_D&=\frac{(K_X\cdot D)}{D^2}
	\\&=\frac{-2(D\cdot C_0)+(2g-2-e)(D\cdot f)}{a(D\cdot C')+(b-\frac{1}{2}ae)(D\cdot f)}
	\\&\le \frac{(2g-2-e)(D\cdot f)}{a(D\cdot C')+(b-\frac{1}{2}ae)(D\cdot f)}
	\\ &\le \frac{|4g-4-2e|}{2b-ae},	
\end{split}
\end{equation*}
where $2b-ae\in\bZ_{>0}$.
In conclusion, $X$ admits the uniform boundedness of $l_C$.

Now let $g=1$. 
Then by \cite[Theorems V.2.12 and V.2.15]{Hartshorne77}, \(e=-1\) holds as we restrict to the case \(e<0\).
Now we may assume that $D=aC_0+bf$ with $a=1, b\ge0$ or $a\ge2,b\ge-\frac{a}{2}$.
Let $D_0=C_0-\frac{1}{2}f$.
Note that $D_0^2=0$ and $K_X\equiv-2(C_0-\frac{1}{2}f)=-2D_0$.
If $b>-\frac{a}{2}$, then $h^1(\mathcal O_X(D))=0$ by \cite[Proposition 3.1]{GP96}.
If $b=-\frac{a}{2}$, then $D=aD_0$ with $a>0$.
As a result, $(K_X-D)D=-(2+a)aD_0^2=0$.
By  the Riemann-Roch Theorem, $(K_X-D)\cdot D=0, q(X)=1$ and Proposition \ref{Serre} imply that 
\begin{equation*}
\begin{split}
h^1(\mathcal O_X(D))&=h^0(\mathcal O_X(D))+h^2(\mathcal O_X(D))+\frac{(K_X-D)\cdot D}{2}-\chi(\mathcal O_X)
\\&\le h^0(\mathcal O_X(D)).
\end{split}
\end{equation*}
In conclusion, $X$ satisfies the BCP.
\end{proof}
\begin{proof}[Proof of Theorem \ref{uniform-thm}]
The proof of (1) follows from \cite[Proof of Lemma 3.1]{Li23}, Proposition \ref{e=0} and  Lemma \ref{e<0}.
As in \cite[Proof of Lemma 3.1]{Li23}, it is shown that $e>0$ if and only if $b(X)>0$.
If $g\le1$,  we may assume that $e\le0$ by \cite[Lemma 3.1]{Li23}.
If  $g=0$, then $e=0$ by \cite[Corollary V.2.13]{Hartshorne77}.
Then $X$ satisfies the BCP by Proposition \ref{e=0}.
If $g=1$, then $X$ satisfies the BCP by Propositions \ref{e=0} and  \ref{e<0}.
If $g\ge2$ and $e>0$, then $X$ satisfies the BCP by \cite[Lemma 3.1]{Li23}.
\end{proof}
\section{Examples}
\label{Sect-Example}
In this section, we present examples of smooth projective surfaces satisfying the BCP. 
We compute the explicit constant $c_X$ from Definition \ref{defn-BCP} for all cases except Mori dream surfaces in Section \ref{MDS-sect}, where an explicit value of \(c_X\) remains undetermined.
\subsection{Surfaces $X$ with $\rho(X)=2,\kappa(X)\ge0$ and $b(X)>0$}
 The second author \cite{Li19} provided a classification of smooth projective surfaces $X$ with $\rho(X)=2$ and two negative curves.
 This result motivated the second author \cite{Li21} to prove the BCP for a smooth projective surface $X$ with $\rho(X)=2$, provided that either $\kappa(X)=1$ and $b(X)>0$ or (ii) $X$ has two negative curves.
 \begin{example}[Surfaces $X$ with $\rho(X)=2,\kappa(X)=1$ and $b(X)>0$]
 Fix an integer $p_g(X) \ge2$, Cox \cite{Cox90} considered the parameter space $\mathcal U$ for relatively minimal elliptic surfaces $f: X\to\bP^1$ with a section, geometric genus $p_g(X)$, and $q(X)=0$.
As stated at the beginning of the first paragraph in \cite[Section 2]{Ulmer17}, it is known that a very general elliptic surface $\pi: X\to\bP^1$ of height $h\ge3$ in $\mathcal U$ has no sections other than the zero section $O$, and $\NE(X)=\bR_{\ge0}[F]+\bR_{\ge0}[O]$ such that $O^2=-h<0, O\cdot F=1, F^2=0$ and $K_X=(h-2)F$.
So $X$ is a surface with $\rho(X)=2,\kappa(X)=1$ and $b(X)>0$.

Without loss of generality, we may assume that $C_1=F$ and $C_2=O$.
Then $\NE(X)=\bR_{\ge0}[C_1]+\bR_{\ge0}[C_2]$.
Now we take a curve $D=a_1C_1+a_2C_2$ with $a_1,a_2\ge0$.
If $D=a_1C_1$ with $a_1>0$, then $l_D=0$ and $(K_X-D)\cdot D=0$.
So by the Riemann-Roch theorem and Proposition \ref{Serre}, we have
\begin{equation*}
\begin{split}
h^1(\mathcal O_X(D))&=h^0(\mathcal O_X(D))+h^2(\mathcal O_X(D))-\chi(\mathcal O_X)
\\&\le h^0(\mathcal O_X(D)).
\end{split}
\end{equation*}
where we use $q(X)=0$.
Note that $D\cdot C_1=a_2>0$ and $D\cdot C_2=a_1>0$.
If $a_1\ge h-2, a_2>0$, then $D^2>0$ and $(K_X-D)D=-(a_1-(h-2))(D\cdot C_1)-a_2(D\cdot C_2)<0$.
We also have $h^1(\mathcal O_X(D))\le h^0(\mathcal O_X(D))$.
If $a_1<h-2, a_2>0$, then $D^2>0$, and  $D\cdot C_2\ge0$ implies that $a_2\le (h-2)h^{-1}$.
As a result,
\begin{equation*}
\begin{split}
(K_X-D)\cdot D&\le K_X\cdot D
\\&\le  (h-2)C_1\cdot (a_1C_1+a_2C_2)
\\ &\le (h-2)^2h^{-1}.
\end{split}
\end{equation*}
So by the Riemann-Roch theorem and Proposition \ref{Serre}, we have
\begin{equation*}
\begin{split}
h^1(\mathcal O_X(D))&=h^0(\mathcal O_X(D))+h^2(\mathcal O_X(D))+\frac{(K_X-D)\cdot D}{2}-\chi(\mathcal O_X)
\\&\le (1+(h-2)^2h^{-1})h^0(\mathcal O_X(D)).
\end{split}
\end{equation*} 
In all, for every curve $C$ on $X$, we have
$$
h^1(\mathcal O_X(C))\le (1+(h-2)^2h^{-1})h^0(\mathcal O_X(C)).
$$
 \end{example}
 There are numerical examples of K3 surfaces $X$ such that $\NE(X)$ is generated by two $(-2)$-rational curves (cf. \cite[Theorem 5.3.2.5]{ADHL15}, see also \cite[Claim 2.12]{Li19}) as follows.
 \begin{example}
 Let $X$ be a K3 surface such that $\NE(X)=\bR_{\ge0}[C_1]+\bR_{\ge0}[C_2]$, and $C_1, C_2$ are $(-2)$-rational curves on $X$.
 Then we have the intersection matrix with respect to $C_1, C_2$:
 $$
\begin{bmatrix} -2 & k \\ k & -2 \end{bmatrix} 
$$
where  $k\in\bZ_{>0}$.
By the proof of Proposition \ref{-K_Xnef}, for every curve $C$ on $X$, we have
$$
h^1(\mathcal O_X(C))\le h^0(\mathcal O_X(C)),
$$
where we use $q(X)=0$ since $X$ is a K3 surface.
 \end{example}
 However, it is unknown to us whether there exists an example of a surface $X$ of general type with $\rho(X)=2$ and two negative curves as follows.
  \begin{question}
Does there exist an example of a minimal smooth projective surface $X$ of general type with $\rho(X)=2$ and two negative curves?
 \end{question}
 Now let $X$ be a smooth projective surface with $\kappa(X)=0$.
If $K_X$ is nef, by the Enriques-Kodaira classification (cf. \cite{Beauville96, BHPV04}), then $X$ is a K3 surface, an Enriques surface, an abelian surface, or a hyperelliptic surface.
By the Proof of Proposition \ref{Prop-Kappa0}, for every curve  $C$ on $X$, we have 
$h^1(\mathcal O_X(C))\le (q(X)+1)h^0(\mathcal O_X(C)).$
If $K_X$ is not nef, then $X$ is a blow-up of a K3 surface or an abelian surface.
There is an example of such $X$ with $\rho(X)=2$ such that $\NE(X)$ is rational polyhedral as follows.
\begin{example}
Let $\pi: X\to S$ be a blow-up of a smooth, very general quartic $S\subseteq\bP^3$ at a very general point $p$, with exceptional divisor $E$ and $H=\pi^*\mathcal O_S(1)$.
 Artebani and Laface showed in  \cite[Section 6]{AL11} that $\kappa(X,-K_X)=-\infty,\kappa(X,H-2E)=0,$ 
 and 
$
\mathrm{Eff}(X)=\langle [H-2E],[E]\rangle.
$
Note that $X$ satisfies the BCP by Proposition \ref{Prop-Kappa0}.
Take a nef and big curve $D=a_1(H-2E)+a_2E$ with each $a_i>0$ by \cite[Proposition 3.1(iii)]{Li21}.
Note that $K_X\equiv E, H^2=4, E\cdot (H-2E)=2, E\cdot D=2a_1-a_2\in\bZ_{>0}$ and $D\cdot (H-2E)=2a_2\in\bZ_{>0}$.
If $a_1>1$, then 
$(K_X-D)\cdot D=(1-a_1)(E\cdot D)-a_2(H-2E)\cdot D\le0$.
If $a_1<1$, then $(K_X-D)\cdot D=E\cdot D-D^2\le 2a_1-a_2< 2$ since $D^2>0$ and $E\cdot D=2a_1-a_2<2$.
Note that $(K_X-E)\cdot E=0$ and $(K_X-(H-2E))\cdot (H-2E)=2$.
So for every curve $C$ on $X$, we have $(K_X-C)\cdot C\le2$.
By the Riemann-Roch Theorem and Proposition \ref{Serre}, for each curve $C$ on $X$, we have 
 \begin{equation*}
 h^1(\mathcal O_X(C))\le (q(X)+1)h^0(\mathcal O_X(C)).
 \end{equation*}
 \end{example}
As in \cite[proof of Lemma 3.6]{Li23},  the second author showed the BCP for one-point blow-up of a K3 surface or an abelian surface $Y$ with $\rho(Y)=1$.
 Moreover, if one would like to study the BCP for a surface $X$ with a higher Picard number, it is worthwhile to consider the BCP for $X$, which is a one-point blow-up of a minimal projective surface $Y$ with $\kappa(Y)=0$ and $\rho(Y)\ge2$ as follows.
\begin{question}
Let $X$ be a one-point blow-up of a minimal projective surface $Y$ with $\kappa(Y)=0$ and $\rho(Y)\ge2$.
Does $X$ satisfy the BCP?
\end{question}
\subsection{Geometrically ruled surfaces}
In this subsection, we use the notation as in \cite[Examples V.2.11.1, 2.11.3, and 2.11.6]{Hartshorne77}.
It is well-known that every geometrically ruled surface $X$ with invariant $e$ over a curve $C$ of genus $g$ corresponds to a locally free sheaf $\mathcal E$ of rank 2 on $C$ such that $X\cong\mathbb P(\mathcal E)$ (cf. \cite[Proposition V.2.2]{Hartshorne77}).
\begin{example}[Rational ruled surfaces]
Let $\mathcal E=\mathcal O_{\bP^1}\oplus\mathcal O_{\bP^1}(-e)$  with $e\ge0$ and $X\cong\mathbb P(\mathcal E)$ over $\bP^1$.
Note that $-K_X=2C_0+(e+2)f$ as in Proposition \ref{Prop-ruled}.
Note that $\NE(X)=\bR_{\ge0}[C_0]+\bR_{\ge0}[f]$ since $C_0^2=-e\le0$ and $f^2=0$.
Then $-K_X\cdot C_0=e\ge0$ and $-K_X\cdot f=2>0$.
Therefore, $-K_X$ is nef.
By the proof of Proposition \ref{-K_Xnef}, for every curve $C$ on $X$, we have
$$
h^1(\mathcal O_X(C))\le h^0(\mathcal O_X(C)),
$$
where we use $q(X)=0$.
\end{example}
\begin{example}[Elliptic ruled surfaces with $e=0$ and $\mathcal E$ decomposable]
Let $\mathcal E=\mathcal O_E\oplus\mathcal O_E$ and $X\cong\mathbb P(\mathcal E)$ over an elliptic curve $E$.
By the proof of Proposition \ref{e=0}, for every curve $C$ on $X$, we have
 $$h^1(\mathcal O_X(C))\le 2h^0(\mathcal O_X(C)),$$
 where we use $q(X)=1$.
\end{example}
\begin{example}[Elliptic ruled surfaces with $e=-1$]
Let $\mathcal E$ be an indecomposable locally free sheaf of rank 2 over an elliptic curve $E$.
Taking $X=\bP(\mathcal E)$, we have an elliptic ruled surface with $e=-1$.
Take a curve $D=aC_0+bf$ as given in Proposition \ref{Prop-ruled}.
Note that $a=1, b\ge0$ or $a\ge2, b\ge-\frac{a}{2}$.
By the proof of Proposition \ref{e<0}, if $D=a(C_0-\frac{1}{2}f)$, then $h^1(\mathcal O(D))\le h^0(\mathcal O_X(D))$.
If $b>-\frac{a}{2}$, then $h^1(\mathcal O_X(D))=0$.
Note that $(K_X-f)f=2$.
By  the Riemann-Roch Theorem, $(K_X-f)\cdot f=2, q(X)=1$ and Proposition \ref{Serre} imply that 
\begin{equation*}
\begin{split}
h^1(\mathcal O_X(f))&=h^0(\mathcal O_X(f))+h^2(\mathcal O_X(f))+\frac{(K_X-f)\cdot f}{2}-\chi(\mathcal O_X)
\\&\le 2h^0(\mathcal O_X(D)).
\end{split}
\end{equation*}
In all, for every curve $C$ on $X$,  we have
$$
h^1(\mathcal O_X(C))\le 2h^0(\mathcal O_X(C)).
$$
\end{example}
\begin{example}[Ruled surfaces with $e>0$ over a curve of genus $g\ge2$]
For a curve $C$ of genus $g\ge2$, take $\mathcal E=\mathcal O_C\oplus\mathcal L$ with $\deg\mathcal L=-e<0$.
Let $C_0\subset X$ be a section and $C_0^2=-e$, and let $f$ be a fiber.
Note that $q(X)=g$ by \cite[Corollary V.2.5]{Hartshorne77}.
 Take a curve $D=a_1C_0+a_2f$ (in which $D\ne C_0,f$).
 Then by \cite[Proposition V. 2.20(a)]{Hartshorne77}, $a_1>0$ and $a_2>a_1e$.
 Note that $K_X\equiv -2C_0+(2g-2-e)f, C_0\cdot D\ge0,(f\cdot D)=a_1>0, (K_X-C_0)\cdot C_0=2g-2-e$ and $(K_X-f)\cdot f=-2$.
 Then $$(K_X-D)\cdot D=-(2+a_1)(C_0\cdot D)+(2g-2-e-a_2)(f\cdot D).$$
 If $a_2\ge 2g-2-e$, then $(K_X-D)\cdot D\le0$.
 If $a_2<|2g-2-e|$, then $a_1<|2g-2-e|e^{-1}$ as $a_1e<a_2$.
 As a result, $(K_X-D)\cdot D\le (2g-2-e)^2e^{-1}$.
By the Riemann-Roch Theorem and Proposition \ref{Serre}, for each curve $C$ on $X$, we have 
 \begin{equation*}
 h^1(\mathcal O_X(C))\le\max\bigg\{2g-1-\frac{e}{2}, 1, g+(2g-2-e)^2e^{-1}\bigg\}h^0(\mathcal O_X(C)).
 \end{equation*}	
\end{example} 
\subsection{Mori dream surfaces}
\label{MDS-sect}
In this subsection, for every Mori dream surface $X$, we cannot give an explicit value $ c_X$ as in Definition \ref{defn-BCP} if $-K_X$ is not nef.
This motivates the following question.
\begin{question}
Let $X$ be a Mori dream surface.
Find an explict value $c_X$ such that $h^1(\mathcal O_X(C))\le c_Xh^0(\mathcal O_X(C))$ for every curve $C$ on $X$.
\end{question}
We first assume that $X$ is a smooth projective surface with $-K_X$ nef and $q(X)=0$.
By the proof of Proposition \ref{-K_Xnef}, for every curve $C$ on $X$, we have 
$$
h^1(\mathcal O_X(C))\le h^0(\mathcal O_X(C)).
$$
Note that $X$ is a Mori dream surface if and only if $-K_X$ is semiample by \cite[Theorem 5.1.2.1]{ADHL15}.
If  $\kappa(X,-K_X)\ge1$, Aretbani and Laface \cite{AL11} proved that $X$ is a Mori dream surface if and only if one of the following holds:
\begin{enumerate}
\item[(i)] $X$ is the minimal resolution of singularities of a Del Pezzo surface with Du Val singularities;
\item[(ii)] $\varphi_{|-mK_X|}$ is an elliptic fibration for some $m>0$ and the Mordell-Weil group of the Jacobian fibration of $\varphi_{|-mK_X|}$ is finite;
\item[(iii)] $X$ is either a K3 surface or an Enriques surface with finite automorphism group $\Aut(X)$.
\end{enumerate}

Testa et al. \cite[Theorem 1]{TVV11} established that a smooth rational surface $X$ is a Mori dream surface provided that $-K_X$ is big. 
They \cite{TVV11} classified surfaces of this type that are blow-ups of $\bP^2$ at distinct points lying on a possibly reducible cubic.
\begin{example}
 Let $X$ be a smooth projective toric surface (cf.  \cite[Definition 3.1.1]{CLS11}).
Note that $\NE(X)$ is rational polyhedral by \cite[Theorem 6.3.20]{CLS11}.
As a result, $X$ satisfies the BNC.
To show the BCP for $X$, we may take a nef curve $C$ on $X$.
Note that $h^1(\mathcal O_X(C))=0$ by the Demazure vanishing theorem  (cf. \cite[Theorem 9.2.3]{CLS11}).
So $X$ satisfies the BCP.
\end{example}
By \cite[Remark 5.1.3.11]{ADHL15}, there are some examples of Mori dream rational surfaces $X$ with $\kappa(X,-K_X)=-\infty$ that can be obtained by constructing smooth surfaces of complexity 1 with large Picard number as follows.
\begin{example}
Let $X$ be a minimal resolution of the singular rational $\mathrm{K}^*$-surface $Y\subseteq\bP(33,22,6,1)$ of equation $x_0^2+x_1^3+x_2^{11}=0$.
Then $X$ is a rational Mori dream surface with $\kappa(X,-K_X)=-\infty$.
\end{example}
In \cite{KL19}, Keum and Lee studied effective, nef and semiample cones on minimal surfaces $X$ of general type with $p_g(X)=0$, and provided examples of minimal surfaces $X$ of general type with $p_g(X)=0$ and $2\le K_X^2\le 9$ which are Mori dream spaces.
In these examples, they further presented an explicit description of their effective cones with all negative curves.
For instance, there is an example of a Mori dream surface $X$ of general type with $\rho(X)=2$ and $b(X)=0$ as follows.
\begin{example}
Let $X$ be a surface isogenous to a higher product of unmixed type with geometric genus $p_g=0$.
 By \cite[Proposition 3.4]{KL19}, $X$ is a Mori dream surface of general type with $\rho(X)=2$ and $b(X)=0$; and the effective cone and nef cone are generated by fibers of $X\to C/G\cong\bP^1$ and fibers of $X\to D/G\cong\bP^1$.
 \end{example}
Now let $\pi: X\to B$ be an elliptic surface whose generic fiber $E/K$(an elliptic curve $E$ over the function field $K=k(B)$).
By \cite[Proposition 5.4]{SS19}, the sections $C$ of $\pi$ are in natural one-to-one correspondence with $K$-rational points $P$ of $E$.
We say a section $C$ of $f: X\to B$ is non-torsion (or torsion) if $P$ is a non-torsion (or torsion) point.
Note that the set $E(K)$ of all sections of $\pi: S\to B$ is a finitely generated abelian group by \cite[Theorem 6.6]{SS19}.
We call $E(K)$  the Mordell-Weil group of the elliptic surface $\pi: X\to B$.
Now we assume that $\chi:=\chi(O_X)>0$ and $X$ has at least one section $C$.
Then $C^2=-\chi<0$ by \cite[Corollary 5.45]{SS19}.
If $C$ is non-torsion, then $X$ has infinitely many $(-\chi)$-curves of genus $g$ by \cite[Proposition 4.1]{Li26}.
This offers some potential for addressing \cite[Question 4.4]{Bauer et al. 2013}, even though this itself remains highly challenging (cf. \cite[Question 4.2]{Li26}).
In other words,  if $\NE(X)$ is rational polyhedral, every section of $\pi$ is torsion.
As a result, $E(K)$ is finite.
This result is stated as follows.
\begin{proposition}
\label{prop-MW}
Let  $X$ be an elliptic surface with $\chi=\chi(\mathcal O_X)>0$ and a section $C$ over a curve of genus $g$.
Then the following statements hold.
\begin{enumerate}
	\item If $C$ is non-torsion, then $X$ has infinitely many $(-\chi)$-curves of genus $g$.
	\item If  $\NE(X)$ is rational polyhedral, then the Mordell-Weil group is finite.
\end{enumerate}
\end{proposition}
As remarked in \cite[Section 7.3]{KL19}, it is an open question whether there exists an example of an elliptic surface $X$ with $\kappa(X)=1$ which is a Mori dream surface.
Now we assume that $\pi: X\to C$ is a relatively minimal elliptic surface with $\kappa(X)=1$ over a smooth curve $C$ of genus $g$.
Let $d=\deg R^1\pi_*\mathcal O_X$.
Note that $\chi(\mathcal O_X)=d$ by \cite[Corollary 16 of Chapter 7]{Friedman98}.
By \cite[Lemma 14 of Chapter 7]{Friedman98}, note that $q(X)=0$ if and only if $g=0$ and $\chi(\mathcal O_X)>0$.
This motivates us to consider whether there exists an example of an elliptic surface $X$ with $\kappa(X)=1$ over $\bP^1$ which is a Mori dream surface as follows.
\begin{question}
\label{Que-MDS-P1}
Does there exist an example of a Mori dream surface $X$ with $\kappa(X)=1$ which admits a relatively minimal elliptic fibration over $\bP^1$?
\end{question}
 \subsection*{\bf Acknowledgments.}
The authors would like to thank the referees for their valuable comments and suggestions.
The research is supported by the Shanghai Sailing Program (No. 23YF1409300).


\begin{thebibliography}{99}

\bibitem{ADHL15}
I. Arzhantsev, U. Derenthal, J. Hausen, and A. Laface,
\emph{Cox Rings}, 
Cambridge Studies in Advanced Mathematics, \textbf{144}, Cambridge University Press, 2015.

\bibitem{AL11}
M. Artebani and A. Laface,
\emph{Cox rings of surfaces and the anticanonical Iitaka dimension},  Adv. Math. \textbf{226} (2011), no. 6, 5252-5267.

\bibitem{Bauer et al. 2012}
T. Bauer, C. Bocci, S. Cooper, S. D. Rocci, M. Dumnicki, B. Harbourne, K. Jabbusch, A. L. Knutsen,
A. K$\mathfrak{\ddot{u}}$ronya, R. Miranda, J. Ro\'e, H. Schenck, T. Szemberg, and Z. Teithler,
\emph{Recent developments and open problems in linear series,}
In: Contributions to Algebraic Geometry, p. 93-140, EMS Ser. Congr. Rep., Eur. Math. Soc., Z\"{u}rich, 2012.

\bibitem{Bauer et al. 2013}
T. Bauer, B. Harbourne, A. L. Knutsen, A. K$\mathrm{\ddot{u}}$ronya, S. M$\mathrm{\ddot{u}}$ller-Stach, X. Roulleau, and T. Szemberg,
\emph{Negative curves on algebraic surfaces,}
Duke Math. J. \textbf{162}(10)(2013), 1877-1894.

\bibitem{Beauville96}
A. Beauville,
\emph{Complex Algebraic Surfaces}, 2nd ed.,
London Math. Soc. Stud. Texts, \textbf{34}, Cambridge University Press, Cambridge, 1996.

\bibitem{BHPV04}
W. P. Barth, K. Hulek, C. A. M. Peters, and A. Van De Ven,
\emph{Compact complex surfaces}, Ergeb. Math. Grenzgeb. \textbf{4}, Springer-Verlag, Berlin, 2004.

\bibitem{CHMR13}
C. Ciliberto, B. Harbourne, R. Miranda, and J. Ro\'e,
\emph{Variations on Nagata's Conjecture},
Clay Math. Proc. \textbf{18} (2013), 185-203.

\bibitem{Ciliberto et al. 2017}
C. Ciliberto, A. L. Knutsen, J. Lesieutre, V. Lozovanu, R. Miranda, Y. Mustopa, and D. Testa,
\emph{A few questions about curves on surfaces,} Rend. Circ. Mat. Palermo. II. Ser. \textbf{66}(2017), no. 2, 195-204.

\bibitem{Cox90}
D. Cox,
\emph{The Noether-Lefschetz locus of regular elliptic surfaces with section and $p_g\ge2$},
Amer. J. Math. \textbf{112} (1990), no. 2, 289-329.

\bibitem{CLS11}
D. Cox, J. Little, and H. Schenck,
\emph{Toric Varieties}, Amer. Math. Soc., 2011.

\bibitem{Friedman98}
R. Friedman,
\emph{Algebraic surfaces and holomorphic vector bundles},
Springer, 1998.

\bibitem{GP96}
F. J. Gallego and B. P. Purnaprajna,
\emph{Normal presentation on elliptic ruled surfaces},
J. Algebra  \textbf{186} (1996), 597-625.

\bibitem{Hartshorne77}
R. Hartshorne,
\emph{Algebraic Geometry,} GTM \textbf{52}, Springer-Verlag, New York, 1977.

\bibitem{Homma82}
Y. Homma,
\emph{Projective normality and the defining equations of an elliptic ruled surface with negative invariant},
Natur. Sci. Rep. Ochanomizu Univ. \textbf{33}(1987), no. 1-2, 17-26.

\bibitem{HK00}
Y. Hu and S. Keel,
\emph{Mori dream spaces and GIT},
Michigan Math. J. \textbf{48}(2000), 331-348.

\bibitem{KL19}
J. Keum and K.-S. Lee,
\emph{Examples of Mori dream surfaces of general type with $p_g=0$},
Adv. Math. \textbf{347} (2019), 708-738.

\bibitem{KM98}
J. Koll\'ar and S. Mori,
\emph{Birational geometry of algebraic varieties},
Cambridge Tracts in Matthmatics, \textbf{134} Cambridge University Press, 1998.

\bibitem{Lazarsfeld04}
R. Lazarsfeld, \emph{Positivity in algebraic geometry, I}, Ergebnisse der
  Mathematik und ihrer Grenzgebiete. \textbf{48}, Springer-Verlag, Berlin, 2004.

\bibitem{Li19}
S. Li,
\emph{A note on a smooth projective surface with Picard number 2,}
Math. Nachr. \textbf{292} (2019), no.~12, 2637-2642.

\bibitem{Li21}
S. Li,
\emph{Bounding cohomology on a smooth projective surface with Picard number 2,}
Commun. Algebra \textbf{49} (2021), no. 7, 3140-3144.

\bibitem{Li23}
S. Li,
\emph{Bounded cohomology property on a smooth projective surface with Picard number two}, Commun. Algebra. \textbf{51} (2023), no. 12, 5235-5241.

\bibitem{Li26}
S. Li,
\emph{Integral Zariski decompositions on smooth projective surfaces}, J. Algebra Appl. \textbf{25} (2026), no. 11, 2650116, 7 pp.

\bibitem{Nagata59}
M. Nagata,
\emph{On the 14th problem of Hilbert},
Amer. J. Math. \textbf{81} (1959), 766-772.

\bibitem{Nagata70}
M. Nagata,
\emph{On self-intersection number of a section on a ruled surface},
Nagoya Math. J. \textbf{37} (1970), 191-196.

\bibitem{Rosoff02}
J. Rosoff,
\emph{Effective divisor classes on a ruled surface}, Asian J. Math.
\textbf{202}(2002), no. 1, 119-124.

\bibitem{Sakai84}
F. Sakai,
\emph{Anticanonical models of rational surfaces},
Math. Ann. \textbf{269} (1984), 389-410.

\bibitem{SS19}
M. Sch\"utt and T. Shioda,
\emph{Mordell-Weil lattices},
 Ergeb. Math. Grenzgeb.  \textbf{70}. Springer, Singapore, 2019. 
 
 \bibitem{TVV11}
 D. Testa, A. V\'arilly-Alvarado, and M. Velasco,
 \emph{Big rational surfaces},
 Math. Ann. \textbf{351} (2011), 95-107.
 
 \bibitem{Ulmer17}
 D. Ulmer,
 \emph{Rational curves on elliptic surfaces},
 J. Algebraic Geom. \textbf{26} (2017), no. 2, 357-377.
 
 
 
 
 
 
 
\end{thebibliography}
\end{document}